\documentclass[12pt]{amsart}
\usepackage {amssymb,latexsym}
\usepackage [all]{xy}
\usepackage{nohyperref}
\usepackage[pageref]{backref}

\newtheorem{theorem}{Theorem}[section]
\newtheorem{lemma}[theorem]{Lemma}
\newtheorem {proposition}[theorem]{Proposition}
\newtheorem {corollary} [theorem] {Corollary}

\theoremstyle{definition}
\newtheorem{definition}[theorem]{Definition}
\newtheorem{example}[theorem]{Example}

\newtheorem{remark}[theorem]{Remark}

\numberwithin {equation}{section}

\renewcommand*{\backref}[1]{}
\renewcommand*{\backrefalt}[4]{%
    \ifcase #1 (Not cited.)%
    \or        (Cited on page~#2.)%
    \else      (Cited on pages~#2.)%
    \fi
    }

\def\cs{{$C^{\ast}$}}
\def\ra{{\rightarrow}}

\def\q{{\mathbb{Q}}}
\def\z{{\mathbb{Z}}}
\def\n{{\mathbb{N}}}
\def\c{{\mathbb{C}}}
\def\r{{\mathbb{R}}}

\def\csr{{C^{\ast}_r}}

\def\H{{\mathcal{H}}}

\def\gh{{(G,H)}}
\def\ghr{{(G_r,H_r)}}
\def\ghpr{{(G',H')}}

\def\f{{\varphi}}

\def\gb{{\overline{G}}}
\def\hb{{\overline{H}}}
\def\ghb{{(\gb,\hb)}}
\def\fh{{_Hf_H}}

\def\la{{\lambda}}

\def\d{{\delta}}
\def\D{{\Delta}}
\def\xx{{\mathfrak{X}}}

\def\wrt{{with respect to }}
\def\G{{\mathcal{G}}}

\def\ba{{\backslash}}
\def\inv{{^{-1}}}
\def\s{{^{\ast}}}

\def\hg{{H\backslash G}}
\def\shat{{\widehat{S}}}
\begin{document}
\title{Length functions and property (RD) for locally compact Hecke pairs}

\author{Vahid Shirbisheh}
\email{shirbisheh@gmail.com}



\begin{abstract}
The purpose of this paper is to study property (RD) for locally compact Hecke pairs. We discuss length functions on Hecke pairs and the growth of Hecke pairs. We establish an equivalence between property (RD) of locally compact groups and property (RD) of certain locally compact Hecke pairs. This allows us to transfer several important results concerning property (RD) of locally compact groups into our setting, and consequently to identify many classes of examples of locally compact Hecke pairs with property (RD). We also show that a reduced discrete Hecke pair $\gh$ has (RD) if and only if its Schlichting completion $\gb$ has (RD). Then it follows that the relative unimodularity is a necessary condition for a discrete Hecke pair to possess property (RD).
\end{abstract}
\maketitle

\section {Introduction}
\label{sec:INTRO}

In noncommutative geometry, Hecke pairs and their associated Hecke \cs-algebras first appeared in the work of Jean-Beno\^{\i}t Bost and Alain Connes in \cite{bc}. Afterwards, various aspects of these objects were studied by numerous authors, see for instance \cite{BCH, hall, klq, tzanev}. We started the study of property (RD) (Rapid Decay) for Hecke pairs in \cite{s1}, see also \cite{s2}. More recently, amenability, weak amenability, Haagerup property (a-T-menability) and property (T) of Hecke pairs have also been studied in \cite{anan,larsen-palma, s3}. In all these works, Hecke pairs have been considered as a generalization of quotient groups, and so capable of many notions and constructions that had been originally invented for groups. In order to delve deeply into the nature of discrete Hecke pairs and Hecke \cs-algebras, one needs to study certain pairs of locally compact groups. With this point of view, Kroum Tzanev, S. Kaliszewski, Magnus B. Landstad, and John Quigg studied a certain completion process of Hecke pairs known as the Schlichting completion in \cite{klq, tzanev}. On the other hand, in \cite{s4}, we observed that many other pairs of locally compact groups also admit a \cs-algebraic formulation, and therefore we developed a generalized setting to study locally compact Hecke pairs. This generalization gave rise to certain results about discrete Hecke pairs too. For instance, we proved that the left regular representation of a discrete Hecke pair is a bounded homomorphism. The main purpose of the present paper is to study the applications of our extended setting of locally compact Hecke pairs in our investigation of property (RD) for these pairs. Our main achievements in this paper are (1) discovering a relationship between property (RD) of locally compact groups and property (RD) of locally compact Hecke pairs, and (2) using the Schlichting completion to relate property (RD) of discrete Hecke pairs to property (RD) of locally compact Hecke pairs. Therefore our work not only applies to non-discrete Hecke pairs, but also it produces several results about property (RD) of discrete Hecke pairs. Besides the main theorems, an important conclusion of our results is that discrete Hecke pairs demonstrate certain features of locally compact groups. Therefore any through study of this objects naturally falls beyond the mere algebraic theory of groups.

As it is explained in \cite{s1}, property (RD) is a basic concept in noncommutative geometry and harmonic analysis which has been defined originally for groups. We showed in \cite{s1} that this property can be defined for Hecke pairs too, and similarly it can be applied to define smooth subalgebras in reduced Hecke \cs-algebra. We refer the reader to \cite{s1} for more details about the importance of studying property (RD) for Hecke pairs and we proceed with an outline of the content of the present paper.

The definition of property (RD) is based on the notion of length functions on groups and on Hecke pairs. Length functions on groups are usually defined by means of isometric free actions of groups on metric spaces. When the action is not free and the stabilizer of a point is a Hecke subgroup of the group, again one can use the action to define a length function on the associated Hecke pair. In Section \ref{sec:lengthfunc}, we explain various constructions and properties of length functions on Hecke pairs. For instance, we show why there always exist proper and locally bounded length functions (analogues to the word length functions on groups) on a given finitely generated discrete Hecke pair. As another conclusion of our discussion, using the Schlichting completion of Hecke pairs, we improve Theorem 2.2 of \cite{s2} which states that property (RD) of discrete Hecke pairs is stable under commensurability of subgroups, see Corollary \ref{cor:commen-RD}.
 
In Section \ref{sec:generalRD}, after defining property (RD) for locally compact Hecke pairs, we discuss this property for Hecke pairs $\gh$ in which $H$ is a cocompact subgroup of $G$. Next we give equivalent definitions of property (RD). Afterwards, we prove the main theorem of this paper, which asserts that when $H$ is a compact subgroup of a locally compact group $G$, the Hecke pair $\gh$ has (RD) if and only if the group $G$ has (RD), see Theorem \ref{thm:compact-sub-RD}. This theorem enables us to transfer many results concerning property (RD) of locally compact groups into our setting, and consequently it expands the classes of examples of Hecke pairs possessing property (RD), see Remark \ref{rem:RD-groups}. Combining this theorem with the Schlichting completion of discrete Hecke pairs, we show that a discrete Hecke pair $\gh$ has (RD) if and only if the totally disconnected locally compact group $\gb$ appearing in the Schlichting completion of $\gh$ has (RD), see Theorem \ref{thm:RD-schlich}. This is in line with the results of \cite{tzanev, anan}, which assert similar statements for amenability and Haagerup property of discrete Hecke pairs. It also helps to find more examples of discrete Hecke pairs possessing property (RD) using rapid decay locally compact groups, see for instance Example \ref{exa:sl2-z1p}. As another application of Theorem \ref{thm:RD-schlich}, we conclude that the relative unimodularity of a discrete Hecke pair $\gh$ is a necessary condition for $\gh$ to possess property (RD). Regarding the similar result for locally compact groups, see Theorem 2.2 of \cite{JiSch}, this shows that even discrete Hecke pairs behave more like locally compact groups. Therefore it supports the idea of studying discrete Hecke pairs in the general context of locally compact groups, see also \cite{s4}. Also it gives us an easy criterion to determine many discrete Hecke pairs which do not have (RD), for example the Bost-Connes Hecke pair, see Example \ref{exa:BC-HP-notRD}.

The last section is devoted to the study of growth rates of Locally compact Hecke pairs and its relationship with property (RD) and amenability. Similar to the case of groups, we observe that for amenable finitely generated discrete Hecke pairs, property (RD) is equivalent to polynomial growth.

In the rest of this section, we recall some notations and results from \cite{s4} concerning locally compact Hecke pairs which are needed for our study. In what follows, $G$ is a locally compact group (possibly discrete) equipped with a right Haar measure $\mu$ and $H$ is a closed subgroup of $G$ equipped with a right Haar measure $\eta$. We denote the modular functions of $G$ and $H$ by $\Delta_G$ and $\Delta_H$, respectively. A pair $\gh$ is called {\it discrete} if the homogeneous space $\hg$ is a discrete space, otherwise it is called {\it non-discrete}. For non-discrete pairs, we assume that $\Delta_G|_H=\Delta_H=1$. Thus there always exists a right $G$-invariant measure $\nu$ on the homogeneous space $H\ba G$ such that the Weil's formula holds:
\[
\int_G f(x)d\mu(x)=\int_{\hg } \int_H f(hy) d\eta(h) d\nu(y), \qquad \forall f\in C_c(G).
\]
The Hecke algebra associated to a (discrete or non-discrete) pair $\gh$ is made of the vector space $\H\gh$ of all compact support continuous complex functions on $\hg$ which are right $H$-invariant, i.e. $f:\hg \to \c$ such that $f(Hxh)=f(Hx)$ for all $x\in G$ and $h\in H$. The multiplication in $\H\gh$ is defined by the following convolution like product:
\[
f\ast g (Hx):=\int_{\hg} f(Hxy\inv)g(Hy)d\nu(Hy),\quad \forall f,g \in \H\gh, Hx\in \hg.
\]

Given a discrete pair $\gh$, for every $x\in G$, we set $L(x):=[H:H_x]$, where $H_x=H\cap xHx\inv$ and $R(x):=L(x\inv)$. Then the pair $\gh$ is called a {\it discrete Hecke pair} if $L(x)< \infty$ for all $x\in G$, equivalently we also say that $H$ is a {\it Hecke subgroup of $G$}. One notes that $L(x)$ (resp. $R(x)$) is the number of distinct left (resp. right) cosets of $H$ in the double coset $HxH$. Equivalently, a discrete pair $\gh$ is a discrete Hecke pair if and only if the characteristic function of every double coset $HxH$ for $x\in G$ belongs to $\H\gh$. In this case $\H\gh$ can also be interpreted as the algebra of all finite support complex functions on the set $G//H$ of all double cosets of $H$ in $G$. The function $\Delta_\gh: G\to \q^+$ defined by $\Delta_\gh (g):= \frac{L(g)}{R(g)}$ is a group homomorphism and is called {\it the relative modular function of the Hecke pair $\gh$}. The involution on the Hecke algebra $\H\gh$ is defined by $f\s(Hx):= \Delta_\gh (x\inv) \overline{f(Hx\inv)}$ for all $Hx\in \hg$. The left regular representation $\la: \H\gh\to B(\ell^2(\hg))$ is defined by $\la(f)(\xi):=f\ast \xi$ for all $f\in \H\gh$ and $\xi\in \ell^2(\hg)$. It was shown in Theorem 5.4 of \cite{s4} that $\la$ is a bounded homomorphism for all discrete Hecke pairs.

By definition a non-discrete pair $\gh$ is called a {\it non-discrete Hecke pair} if (1) the left regular representation $\la: \H\gh\to B(L^2(\hg))$ is a well defined homomorphism, and (2) for every $Hx\in \hg$, there is some $f\in \H\gh$ such that $f(Hx)\neq 0$. We define the involution of $\H\gh$ by setting $f\s(Hx):= \Delta_G (x\inv) \overline{f(Hx\inv)}$ for all $Hx\in \hg$. It was shown in Theorem 4.1 and Corollary 6.2 of \cite{s4} that when $H$ is a compact or cocompact subgroup of a locally compact group $G$, the pair $\gh$ is a Hecke pair. Also, by Theorem 4.1 of \cite{s4}, the left regular representation $\la$ is a bounded homomorphism when $H$ is a compact subgroup of $G$. The definition of non-discrete Hecke pairs is flexible enough to include discrete Hecke pairs of the form $\gh$, where either $H$ is a Hecke subgroup of discrete group $G$ or $H$ is a compact open subgroup of a locally compact group $G$. Other discrete Hecke pairs $\gh$ also can be included in the definition of non-discrete Hecke pairs if we ignore their topology and consider $G$ and $H$ simply as discrete groups.

Given a (discrete or non-discrete) Hecke pair $\gh$, the norm closure of the image of the left regular representation in $B(L^2(\hg))$ is called the {\it reduced Hecke \cs-algebra of the Hecke pair $\gh$} and is denoted by $C_r^\ast \gh$. The following is parts of Proposition 2.13 of \cite{s4} and is applied to realize more Hecke pairs by means of reduction:

\begin{proposition}
\label{prop:normalcomponent}
Let $\gh$ be a pair and let $N$ be a normal closed subgroup of $G$ contained in $H$. Set $G':=\frac{G}{N}$ and $H':=\frac{H}{N}$.
\begin{itemize}
\item [(i)] Assume that the pair $\gh$ satisfies the assumptions $\Delta_G|_H=\Delta_H=1$, then the pair $\gh$ is a Hecke pair if and only if the pair $(G',H')$ is a Hecke pair. In this case, the Hecke algebras $\H\gh$ and $\H(G',H')$ are isomorphic.
\item[(ii)] With the assumptions of item (iii), the left regular representation of $\H\gh$ is bounded if and only if the left regular representation of $\H(G',H')$ is bounded. Moreover, the \cs-algebras $\csr\gh$ and $\csr(G',H')$ are isomorphic.
\end{itemize}
\end{proposition}

Given a discrete Hecke pair $\gh$, set $K_\gh:=\bigcap_{x\in G} xHx\inv$. Then $K_\gh$ is a normal (closed) subgroup of $G$ and the pair $(\frac{G}{K_\gh}, \frac{N}{K_\gh})$ is a discrete Hecke pair called the {\it reduction of $\gh$} and is denoted by $\ghr$. The Hecke pair $\gh$ is called {\it reduced} if it equals to its reduction, i.e. $ K_\gh=\{e\}$. The Schlichting completion is a process to associate a totally disconnected locally compact group $\gb$ with a given reduced discrete Hecke pair $\gh$ such that $G$ is embedded densely in $\gb$, and more importantly, the closure of $H$ in $\gb$, denoted by $\hb$, is a compact open subgroup of $\gb$. Then the discrete Hecke pair $\ghb$ is called the {\it Schlichting completion of $\gh$}.

\begin{lemma}
\label{lem:bijective-Schlich} (\cite{klq}, Proposition 4.9) Let $\gh$ be a reduced discrete Hecke pair and let $\ghb$ be its Schlichting completion. Then the following statements hold:
\begin{itemize}
\item[(i)] The mapping $\alpha: \hg\to \hb\ba\gb$ (resp. $\alpha':G/H\to \gb / \hb$), defined by $Hg\mapsto \hb g$ (resp. $gH\mapsto g\hb$) for all $g\in G$ is a $G$-equivariant bijection. In particular, $\alpha$ induces an isometric isomorphism between Hilbert spaces $\ell^2(\hg)$ and $\ell^2(\hb \ba \gb)$.
\item[(ii)] The mapping $\beta: G//H \to \gb// \hb$, defined by $HgH\mapsto \hb g\hb$ for all $g\in G$ is a bijection.
\item[(iii)] The mapping $\beta$ commutes with the convolution product, and therefore it induces an isometric isomorphism between the Hecke algebras $\H\gh$ and $\H\ghb$, with respect to the corresponding $\ell^1$-norms.
\end{itemize}
\end{lemma}

Therefore the reduced Hecke \cs-algebras of $\gh$ and $\ghb$ are isomorphic. In the notation of the above lemma, let $\Gamma$ be a subgroup of $G$ containing $H$ and let $\overline{\Gamma}^G$ be its closure in $\gb$. Then the same correspondences hold between $H\ba \Gamma$ and $\hb\ba \overline{\Gamma}^G$ and between $\Gamma// H$ and $\overline{\Gamma}^G // \hb$. We conclude this section by an example for the Schlichting completion of a Hecke pair.

\begin{example}
\label{exa:SL2zp} Let $p$ be a prime number and consider the Hecke pair $(SL_2(\z [1/p]), SL_2(\z))$. Then $K_{(SL_2(\z [1/p]), SL_2(\z))}=\{I, -I\}$, where $I$ is the $2\times 2$ identity matrix. Thus the reduction of the above Hecke pair is $(PSL_2(\z [1/p]), PSL_2(\z))$. It was shown in Example 11.8 of \cite{klq} that the Schlichting completion of the latter Hecke pair is $(PSL_2(\q_p), PSL_2(\z_p))$ which is also the reduction of $(SL_2(\q_p), SL_2(\z_p))$. Therefore all these four Hecke pairs give rise to isomorphic Hecke algebras and isomorphic reduced Hecke \cs-algebras.
\end{example}

\section {Length functions on locally compact Hecke pairs}
\label{sec:lengthfunc}

The definition of property (RD) is based on certain geometric and analytic notions defined on groups. The geometric aspects of property (RD) rely on the notion of length functions on groups (and Hecke pairs). In this section we collect definitions and lemmas concerning length functions which are necessary for our study of property (RD).

A {\it length function on a locally compact group $G$} is a Borel function $l:G\ra [0,\infty[$ such that for all $g, h\in G$, we have
\begin{itemize}
\item $l(e)=0$,
\item $l(g)=l(g\inv)$, and
\item $l(gh)\leq l(g)+l(h)$.
\end{itemize}
Then the set
\[
N_l:=\{ g\in G; l(g)=0\}
\]
is a subgroup of $G$ and is called the {\it kernel of $l$}. The length function $l$ is called {\it closed} if its kernel is a closed subgroup of $G$. A {\it length function on a Hecke pair $\gh$} is a length function on $G$ such that $H\subseteq N_l$. In this case, $l$ is a bi-$H$-invariant function on $G$. It follows that when the Hecke pair $\gh$ is discrete, $l$ is a continuous function, and so a closed length function.

In \cite{j2}, a length function is assumed to be continuous. Following \cite{JiSch, chatpitsaloff}, we weaken this assumption by assuming $l$ to be only a Borel function in order to include word length functions associated with compact generating subsets of $G$.

\begin{definition}
\label{def:bounded-proper}
In the following $G$ is a locally compact group and $\gh$ is a Hecke pair.
\begin{itemize}
\item[(i)] A length function $l$ on $G$ (resp. $\gh$) is called {\it locally bounded} if $l$ is bounded on every compact subset of $G$ (resp. $\hg$).
\item[(ii)] A length function $l$ on $G$ (resp. $\gh$) is called {\it proper} if $l\inv([0,n])$ is relatively compact in $G$ (resp. $\hg$).
\item[(iii)] A subset $S$ of $G$ is called a {\it generating set of the group $G$} (resp. {\it the Hecke pair $\gh$}) if $G=\bigcup_{n\in\n} \shat^n$ (resp. $\hg=\bigcup_{n\in\n} H\shat^n$), where $\shat:= S\cup S\inv \cup\{e\}$. The group $G$ (resp. the Hecke pair $\gh$) is called {\it compactly generated} if it has a compact generating set. A compactly generated group (resp. Hecke pair) is called {\it finitely generated} if the topology is discrete.
\item[(iv)] Let $S$ be a generating set of $G$. The function
\[
l_S(g):= \left\{ \begin{array}{lcr} \min\{ n; g\in \shat^n\}&  & e\neq g\in G \\
0 & & g=e
\end{array} \right.
\]
is called the {\it word length function on $G$ associated with $S$}.
\item[(v)] Let $l_1$ and $l_2$ be two length functions on $G$ (resp. $\gh$). We say that {\it $l_1$ dominates $l_2$} if there are constants $c_0,c_1\geq 0$ such that $l_2(g)\leq c_1 l_1(g)+ c_0$ for all $g\in G$. If $l_1$ and $l_2$ dominate each other, we call them {\it equivalent}.
\end{itemize}
\end{definition}

For basic properties and characterizations of compactly generated groups, we refer to Section 2.C of \cite{cor-harpe} or Theorem 6.11 and Corollary 6.12 of \cite{stroppel}. For instance, every almost connected locally compact group is compactly generated. We recall that a locally compact group $G$ is called {\it almost connected}, if $G/G_0$ is a compact group, where $G_0$ is the connected component of the identity.

\begin{remark}
\label{rem:wordlength}
Given a compactly (or finitely) generated Hecke pair $\gh$ with a generating set $S$, we cannot imitate Definition \ref{def:bounded-proper}(iv) to define a word length function on $\gh$. As an alternative method, we can use length functions defined geometrically, as it is explained in the next remark. When the subgroup $H$ in the Hecke pair $\gh$ is compact, there is another method to define a length function on $\gh$, see Lemma \ref{lem:comp-length}. Also, in Example \ref{exa:lengthfunctions}(ii), for a given discrete Hecke pair $\gh$, we define a length function on $\gh$ using its algebraic structure.
\end{remark}

Let $l$ be a length functions on a locally compact group $G$. Define a function $d_l:G\times G\to [0,\infty[$ by setting $d_l(x,y):=l(x\inv y)$ for all $x,y\in G$. It is straightforward to check that $d_l$ is a {\it pseudo-metric} on $G$, that is for every $x,y,z\in G$, we have
\begin{itemize}
\item[(i)] $d_l(x,x)=0$,
\item[(ii)] $d_l(x,y)=d_l(y,x)$,
\item[(iii)] $d_l(x,z)\leq d_l(x,y) + d_l(y,z)$.
\end{itemize}
This function is also left $G$-invariant;
\begin{itemize}
\item[(iv)] $d_l(gx,gy)=d_l(x,y)$, for all $x,y,g\in G$.
\end{itemize}
If we assume that the kernel of $l$ is trivial, i.e. $N_l=\{e\}$, then $d_l$ is a metric on $G$, that is in addition to Items (i), (ii) and (iii), for every $x,y\in G$, we have
\begin{itemize}
\item[(v)] $d_l(x,y)=0 \Leftrightarrow x=y$.
\end{itemize}
This last condition is equivalent to saying that the topology defined by $d_l$ is Hausdorff.

Conversely, given a left $G$-invariant metric $d$ on $G$, we can define a length function $l_d:G\to [0,\infty[$ by specifying an element $x_0$ of $G$ and setting
\[
l_d(g):=d(x_0,gx_0), \quad \forall g\in G.
\]
Since the action of $G$ on itself is transitive and $d$ is left invariant, $l_d$ does not depend on $x_0$. The kernel of $l_d$ is always trivial. Moreover, we always have $d_{l_d}=d$, and if $N_l=\{e\}$, then we have $l_{d_l}=l$. In the following remark, we see that the correspondence between length functions on $G$ with trivial kernel and metrics on $G$ can be generalized to arbitrary length functions (including length functions on Hecke pairs) and pseudo-metrics on $G$.

\begin{remark}
\label{rem:metric-length}
\begin{itemize}
\item[(i)]  Assume $(X,d)$ is a metric space and $G$ acts continuously on $X$ from left by isometries, so $d(gx, gy)=d(x,y)$ for all $g\in G$ and $x,y\in X$. Fix a point $x_0\in X$ and define $l_{d,x_0}(g):=d(x_0,gx_0)$ for all $g\in G$. Then $l_{d,x_0}$ is a length function on $G$. One notes that $l_{d,x_0}$ depends on $x_0$, and its orbit. The kernel of $l_{d,x_0}$ is the stabilizer  subgroup $G_{x_0}$ of $x_0$, and so it is always a closed subgroup of $G$.

    If the action of $G$ on $X$ is transitive and free, then $l_{d,x_0}$ does not depend on $x_0$. In this case, there is a bijection between $G$ and $X$, which is defined by fixing a point $x_0\in X$, and we can define a metric $d'$ on $G$ using $d$ by setting $d'(g_1,g_2):=d(g_1x_0,g_2x_0)$ and one checks that $l_{d,x_0}=l_{d'}$, or equivalently $d'=d_{l_{d,x_0}}$.
\item[(ii)] The parallelism between the notions of length functions and left invariant pseudo-metric structures on $G$ (or generally any set with an action of $G$) motivates the following definitions: Given a topological space $X$, we call a pseudo-metric $d$ on $X$ {\it locally bounded}  if for every $x_0\in X$ there exists a neighborhood of $x_0$ with bounded diameter, and we call $d$ {\it proper} if for every $x_0\in X$ the map $X\to [0,\infty[$ defined by $x\mapsto d(x_0,x)$ is proper. Then one checks that every locally bounded and proper length function on a group gives rise to a left invariant, locally bounded and proper pseudo-metric on $G$, and vice versa.
\end{itemize}
\end{remark}
The above discussion applies to any Hecke pair $\gh$ where $G$ is a group of isometries of some metric space $X$ and $H$ is contained in the stabilizer of some point of $X$. The following example describes an instance coming from Lie theory. For examples of discrete Hecke pairs coming from groups acting on trees, see \cite{BLRW2009} and Example 4.4 of \cite{s4}.

\begin{example}
\label{exa:Length-sym-spaces}
In Chapter VI of \cite{helgason}, it was shown that every connected semisimple Lie group $G$ with finite center has a maximal compact subgroup $K$ such that for a suitable Riemannian structure on $G/K$, the natural action of $G$ on the homogeneous space $G/K$ (by left multiplication) identifies $G$ with $I_0(G/K)$, the identity component of the Lie group of all isometries of $G/K$. Moreover, $K$ is the stabilizer subgroup of some point $o\in G/K$. Thus by applying Remark \ref{rem:metric-length}(i), we can define a length function on the Hecke pair $(G,K)$. It is also known that every compact subgroup of $G$ is contained in a maximal compact subgroup of $G$ and every maximal compact subgroup of $G$ is a conjugate of $K$. Therefore this construction applies to all Hecke pairs $\gh$, where $H$ is a compact subgroup of $G$.
\end{example}

The study of groups by means of their actions on geometric structures (mostly metric spaces) is a vivid trend in group theory which lies in the domain of geometric group theory. Here, we content ourselves with some definitions and lemmas which are needed for our study and post pone the study of Hecke pairs by means of geometric notions to another time. The interested reader can find most of these materials as well as more comprehensive discussion of geometric group theory for locally compact groups in \cite{cor-harpe}.
\begin{definition}
\label{def:metricgeometry1}
Let $(X_1, d_1)$ and $(X_2, d_2)$ be two pseudo-metric spaces and let $\f:X_1\to X_2$ be a mapping.
\begin{itemize}
\item[(i)] The mapping $\f$ is called {\it large scale Lipschitz} if there are constants $c_0\geq 0$ and $c_1>0$ such that
\[
d_2(\f(x),\f(y))\leq c_1d_1(x,y)+c_0, \quad \forall x,y\in X_1.
\]
It is called {\it Lipschitz} if the above inequality holds for $c_0=0$.
\item[(ii)] The mapping $\f$ is called {\it large scale bilipschitz} if there are constants $c_0, c_0^{\prime}\geq 0$, $c_1, c_1^{\prime}>0$ such that
\[
c_1^{\prime} d_1(x,y)-c_0^{\prime}\leq d_2(\f(x),\f(y))\leq c_1d_1(x,y) +c_0, \quad \forall x,y\in X_1.
\]
It is called {\it bilipschitz} if the above inequalities hold for  $c_0, c_0^{\prime}= 0$.
\item[(iii)] The mapping $\f$ is called {\it large scale bilipschitz equivalence} (resp. {\it bilipschitz equivalence}) if it is bijective and large scale bilipschitz (resp. bilipschitz).
\end{itemize}
\end{definition}

\begin{lemma}
\label{lem:len-geom1} Let $l_1$, $l_2$ be two length functions on $G$ and let $d_1$, $d_2$ be their associated pseudo-metrics, respectively.
\begin{itemize}
\item[(i)] The length function $l_1$ dominates $l_2$ if and only if the identity map $id:(G,d_2) \to (G,d_1)$ is a large scale Lipschitz map. Therefore the length functions $l_1$ and $l_2$ are equivalent if and only if the identity map, $id$, is a large scale bilipschitz equivalence.
\item[(ii)] Let $l_1$ be the length function associated with a compact generating set $S_1$ of $G$ and let $l_2$ be a locally bounded length function on $G$. Then $l_1$ dominates $l_2$.
\item[(iii)] Let $l_1$ and $l_2$ be length functions associated with two compact generating sets $S_1$ and $S_2$ of $G$, respectively. Then they are equivalent.
\item[(iv)] Let $l_1$, $l_2$, $S_1$ and $S_2$ be as described in (iii). Then $id:(G,d_2) \to (G,d_1)$ is a bilipschitz equivalence.
\end{itemize}
\end{lemma}

\begin{proof}
\begin{itemize}
\item[(i)] It follows from the correspondence between length functions on $G$ and pseudo-metrics on $G$.
\item[(ii)] Since $l_2$ is locally bounded and $S_1$ is compact, there is some $c\geq 0$ such that $c\geq l_2(s)$ for all $s\in S_1\cup S_1\inv$. By definition, for every $e\neq g\in G$, there is $n\in \n$ such that $l_1(g)=n$ and there are some $s_1,\cdots, s_n\in (S_1\cup S_1\inv)-\{e\}$ such that $g=s_1\cdots s_n$. For every $1\leq i\leq n$, we have $l_1(s_i)=1$. Thus we have
    \begin{eqnarray*}
    l_2(g)&=&l_2(s_1\cdots s_n)\leq l_2(s_1)+\cdots+l_2(s_n)\\
    &\leq& cn=cl_1(g).\\
    \end{eqnarray*}
\item[(iii)] It follows from (ii) and the fact that every length function associated with a compact generating set is locally bounded.
\item[(iv)] It follows from the proof of (ii).
\end{itemize}
\end{proof}
\begin{remark}
\label{rem-lengths} Let $l_1$ and $l_2$ be two length functions on $G$.
\begin{itemize}
\item[(i)] If $l_1$ and $l_2$ are equivalent, then $l_1$ is proper (resp. locally bounded) if and only if $l_2$ is proper (resp. locally bounded).
\item[(ii)] If $G$ is compactly generated, then it possesses a compact generating set which is also a symmetric neighborhood of $\{e\}$, for the proof see Page 514 of \cite{chatpitsaloff}.
\item[(iii)] It follows immediately that every length function associated with a compact generating set is proper and locally bounded.
\end{itemize}
\end{remark}

Whenever $H$ is a compact subgroup of a locally compact group $G$, the following lemma is a useful tool to define length functions on the Hecke pair $\gh$ using length functions defined on $G$. Its proof is a modification of the proof of Lemma 2.1.3 of \cite{j2}.

\begin{lemma}
\label{lem:comp-length}
Let $H$ be a compact subgroup of a locally compact group $G$. If $l$ is a length function on $G$ which is bounded on $H$, then there exists a length function $l'$ on $G$ such that the following statements hold:
\begin{itemize}
\item[(i)] The length functions $l$ and $l'$ are equivalent.
\item[(ii)] The kernel of $l'$ contains $H$, i.e. $l'$ is a length function on the Hecke pair $\gh$.
\item[(iii)] If the length function $l$ is locally bounded (resp. proper), then $l'$ is locally bounded (resp. proper), both as a length function on $G$ and as a length function on the Hecke pair $\gh$.
\end{itemize}
\end{lemma}
\begin{proof}
Define $l_1:G\to [0,\infty[$ by
\[
l_1(g):=\int_H l(h gh\inv) d\eta(h), \quad \forall g\in G.
\]
One observes that $l_1$ is a length function on $G$ such that
\begin{equation}
\label{eqn:Hlength1}
l_1(hg)=l_1(gh), \quad \forall g\in G, h\in H.
\end{equation}
Also, $l_1$ is equivalent to $l$. This follows from the boundedness of $l$ on $H$ and the following inequalities:
\begin{equation}
\label{eqn:Hlength2}
l_1(g) \leq \eta(H) l(g) +2\int_H l(h)d\eta(h), \quad \forall g\in G,
\end{equation}
\[
l(g)\leq \frac{1}{\eta(H)} l_1(g)+\frac{2}{\eta(H)}\int_H l(h)d\eta(h), \quad \forall g\in G.
\]
It follows from (\ref{eqn:Hlength2}) that $l_1$ is bounded on $H$ too. Now, define $l':G\to [0,\infty[$ by
\[
l'(g):=\inf_{h,h' \in H} l_1(h gh'), \quad \forall g\in G.
\]
Using (\ref{eqn:Hlength1}), one observes that $l'(g)=\inf_{h \in H} l_1(h g)=\inf_{h \in H} l_1(gh)$ for all $g\in G$. Therefore $l'$ is a length function on $G$. Moreover, $l'(g)\leq l_1(g)$ and $l_1(g)\leq l_1(gh)+l_1(h)$ for all $g\in G$ and $h\in H$. By boundedness of $l_1$ on $H$ and these inequalities, $l_1$ and $l'$ are equivalent. Assertions (i), (ii) and (iii) follow from the above discussion.
\end{proof}
Besides the zero length function and constructions discussed in Remark \ref{rem:metric-length} and Lemma \ref{lem:comp-length}, the following examples are useful too:
\begin{example}
\label{exa:lengthfunctions}
\begin{itemize}
\item[(i)] Let $\gh$ be a Hecke pair. Define $l:G\to \{0,1\}$ by
    \[
l(g):=\left\{ \begin{array}{lll} 0&\quad &g\in H\\
1& \quad& \text{otherwise}  \end{array}\right.
    \]
    This is a length function on $G$ which is proper if and only if $H$ is cocompact in $G$.
\item[(ii)] Let $\gh$ be a relatively unimodular discrete Hecke pair. Let $L$ be as defined in Section \ref{sec:INTRO}. Define $l_c:G\to [0,\infty[$ by setting $l_c(g):=\log(L(g))$. One checks that $l_c(g)=0$ if and only if $g$ belongs to the normalizer of $H$ in $G$, i.e. $g\in N_G(H)$. Due to relative unimodularity of $\gh$, one also checks that $l_c(g)=l_c(g\inv )$ for all $g\in G$. For given $x,y \in G$, let $m=L(x)$, $n=L(y)$, $HxH= \cup_{i=1}^m x_i H$ and $HxH= \cup_{j=1}^n y_j H$. Then we have
    \begin{eqnarray*}
    HxyH&\subseteq& (HxH)( HyH) =\left( \bigcup_{i=1}^m x_i H\right)HyH\\
    &=&\bigcup_{i=1}^m x_i \left( \bigcup_{j=1}^n  y_j H\right) = \bigcup_{i=1}^m \bigcup_{j=1}^n x_i y_j H.
    \end{eqnarray*}
    This shows that $L(xy)\leq L(x) L(y)$, and so $l_c(xy)\leq l_c(x)+l_c(y)$. Therefore $l_c$ is a length function on the Hecke pair $\gh$. We call it the {\it characteristic length function of $\gh$}.

    By using $L(g)R(g)$ instead of $L(g)$ in the definition of $l_c$, one can drop the condition of relative unimodularity. However, since we are dealing with property (RD) of Hecke pairs and relative unimodularity is a necessary condition for (RD), see Corollary \ref{cor:unimod-RD}, this condition is not an important restriction.

    One observes that $l_c$ is always locally bounded. Moreover, it follows from Theorem 3 of \cite{berlenstra} that $l_c$ is bounded if and only if $H$ is a nearly normal subgroup of $G$, see also the proof of Proposition 3.10(i) of \cite{s4}. Therefore it is an interesting problem to find a condition which is equivalent to (or at least implies that) $l_c$ is a proper length function. Surprisingly, one notes that such a sufficient condition has already appeared in a different context in the literature, see Proposition 2 of \cite{BCH}.
\end{itemize}
\end{example}

While dealing with length functions on discrete Hecke pairs, it is necessary to pay attention to certain easy but important details.
\begin{remark}
\label{rem:len-schlich-reduced}
\begin{itemize}
\item[(i)] Let $\gh$ be a discrete Hecke pair and let $(G_r,H_r)$ be its reduction. For every length function $l$ on $\gh$, we define a length function $l_r$ on the Hecke pair $(G_r,H_r)$ by $l_r(gK_\gh):=l(g)$ for all $gK_\gh\in G/K_\gh$. Then the mapping $l\mapsto l_r$ is a natural bijective correspondence between the set of length functions on $\gh$ and the ones on $\ghr$. It is clear that this correspondence preserves local boundedness, properness and equivalence of length functions. Moreover, if $S$ is a generating set for $\gh$, then $S':=\{sK_\gh, s\in S\}\subseteq G_r$ is a generating set for $\ghr$. One also notes that the same comment applies to any other normal subgroup $K$ of $G$ which is contained in $H$.
\item[(ii)] Let $\gh$ be a reduced discrete Hecke pair and let $\ghb$ be its Schlichting completion. If $\bar{l}:\gb \to [0,\infty[$ is a length function on the Hecke pair $\ghb$, then its restriction to $G$, say $l$, is clearly an algebraic length function on $G$. We only need to check that $l$ is a Borel function. But this follows from the fact that $H$ is open in $G$ and $l$ is constant on every left coset of $H$ in $G$.

    Conversely, let $l:G\to [0,\infty[$ be a length function on $\gh$. By Lemma \ref{lem:bijective-Schlich}(i), we can define $\bar{l}:\gb \to [0,\infty[$ by $\bar{l}(x):= l(g_x)$ for all $x\in \gb$, where $g_x\in G$ is chosen such that $\hb x=\hb g_x$. Using Lemma \ref{lem:bijective-Schlich}(i) and the fact that $l$ is a bi-$H$-invariant function on $G$, $\bar{l}$ is a well defined bi-$\hb$-invariant function on $\gb$ such that $\bar{l}(x)=0$ for all $x\in \hb$. It follows that $\bar{l}$ is a continuous function, and so a Borel function. Since $HgH\mapsto Hg\inv H$ is a bijection from $G//H$ onto $G//H$ and by Lemma \ref{lem:bijective-Schlich}(ii), we have $\bar{l}(x)=\bar{l}(x\inv)$ for all $x\in \gb$. For given $x,y\in \gb$, pick $g_x, g_{y\inv}\in G$ such that $\hb x=\hb g_x$ and $\hb y\inv= \hb g_{y\inv}$. Thus $\hb xy \hb= \hb g_x g_{y\inv}^{-1} \hb$, and therefore
\[
\bar{l}(xy) = l(g_x g_{y\inv}^{-1} )\leq l(g_x) +l(g_{y\inv})= \bar{l}(x) + \bar{l}(y\inv)=\bar{l}(x) + \bar{l}(y).
\]
    This shows that $\bar{l}$ is a length function on the Hecke pair $\ghb$. Therefore there is a natural bijection between the sets of length functions on the Hecke pairs $\gh$ and $\ghb$ as above.

    Finally, one notes that if $S$ is a finite generating set for $\gh$, then $\hb S$ is a compact generating set for $\gb$. Conversely, if $\gb$ is compactly generated, then the discrete Hecke pair $\gh$ is finitely generated.
\end{itemize}
\end{remark}
We shall see in Remark \ref{rem:RDloc-bound-proper} that for property (RD), we need locally bounded and proper length functions. But the characteristic length function of a Hecke pair can easily fail to be proper, for instance when $[N_G(H):H]=\infty$. So we need another method to ensure the existence of locally bounded and proper length functions on discrete Hecke pairs, at least when they are finitely generated.

\begin{lemma}
\label{lem:lengthexist}
Let $\gh$ be a finitely generated discrete Hecke pair. Then there exists a locally bounded and proper length function on $\gh$.
\end{lemma}
\begin{proof}
Using the Schlichting completion and the above remark, we can assume that $H$ is a compact open subgroup of $G$. Then $\gh$ being finitely generated amounts to $G$ being compactly generated. Thus there exists a locally bounded and proper length function $l$ on $G$, see Remark \ref{rem-lengths}(iii). By Lemma \ref{lem:comp-length}, we can replace $l$ with another locally bounded and proper length function, say $l'$, (equivalent to $l$)  such that the kernel of $l'$ contains $H$.
\end{proof}

Let $H$ and $K$ be two subgroups of a group $G$. They are called {\it commensurable} if $H\cap K$ is a finite index subgroup of both $H$ and $K$.

\begin{lemma}
\label{lem:commen-length}
Let $H$ and $K$ be two open Hecke subgroups of a group $G$ and let they be commensurable. Given a length function $l$ on one of the Hecke pairs $\gh$ or $(G,K)$, one can find a length function $l'$ on $G$ such that it is equivalent to $l$ and the kernel of $l'$ contains both $H$ and $K$.
\end{lemma}
\begin{proof}
Without loss of generality, we can assume that $K$ is a finite index subgroup of $H$. If $l$ is a length function on $(G,K)$, then by an argument built on the Schlichting completion, similar to the proof of Lemma \ref{lem:lengthexist}, we obtain a length function $l'$ equivalent to $l$ such that the kernel of $l'$ contains both $K$ and $H$. The other implication is obvious.
\end{proof}

Using the above lemma, we can improve the statement of Theorem 2.2 of \cite{s2} as follows:
\begin{corollary}
\label{cor:commen-RD} Let $H$, $K$ and $G$ be as Lemma \ref{lem:commen-length}. Then the Hecke pair $\gh$ has property (RD) if and only if the Hecke pair $(G,K)$ has property (RD).
\end{corollary}

\section{Property (RD) for locally compact Hecke pairs}
\label{sec:generalRD}

In this section we state and prove our main theorems concerning property (RD) and their consequences.

\begin{definition}
\label{def:propertyRD} Let $\gh$ be a Hecke pair.

\begin{itemize}
\item[(i)] Given $s>0$, every locally bounded length function $l$ on $\gh$ defines a {\it weighted $L^2$-norm on $\H\gh$} as follows:
    \[
    \|f\|_{s,l}:=\left(\int_{H\ba G} |f(x)|^2 (1+l(x))^{2s} \right)^{1/2}, \quad \forall f\in \H\gh,
    \]
\item[(ii)] We say that $\gh$ has {\it property (RD)} if there exist a locally bounded length function $l$ on $\gh$ and real numbers $s,c>0$ such that
    \[
    \|\la(f)\|\leq c\|f\|_{s,l}, \quad \forall f\in \H\gh,
    \]
    where $\|\la(f)\|$ is the {\it convolution norm of $f$} defined using the left regular representation.
\end{itemize}
\end{definition}
In the following remark we explain why we usually consider locally bounded and proper length functions in our study of property (RD).

\begin{remark}
\label{rem:RDloc-bound-proper}
\begin{itemize}
\item[(i)] The locally boundedness of a length function $l$ is required in order to insure that the weighted $L^2$-norm $\|f\|_{s,l}$ is well defined for all $f\in \H\gh$.
\item[(ii)] Boundedness of length functions imposes a substantial restriction on groups having property (RD). In fact, if a locally compact group $G$ has (RD) with respect to a bounded length function $l$, then the space $L^2(G)$ is closed under the convolution product, and so it becomes an algebra. However, when $G$ is an abelian locally compact group, $L^2(G)$ equipped with the convolution product is an algebra if and only if $G$ is compact, see \cite{rajago}. This shows that for the purpose of studying property (RD) for non-compact groups, we have to consider non-bounded length functions. One way to avoid bounded length functions is to consider proper length functions. The reason is that if $l$ is a bounded and proper length function on a locally compact group $G$, then $G$ has to be compact, see also \cite{cor-lengths} for other clues that boundedness and properness of length functions are opposite notions in some sense.
\end{itemize}
\end{remark}
The boundedness of the left regular representation is a necessary condition to study various features of property (RD) for Hecke pairs in analogy with property (RD) for groups, see the following proposition and Proposition \ref{prop:RDdefinitions}.

\begin{proposition}
\label{prop:cofiniteRD}
Let $H$ be a cocompact subgroup of $G$. Then the Hecke pair $\gh$ has (RD) whenever the left regular representation $\la:\H\gh\to B(L^2(\hg))$ is bounded.
\end{proposition}
\begin{proof}
Let $C>0$ be the norm of $\la$ and let $l$ be the zero length function. Then for given $f\in H\gh$ and $s>0$, we have
\[
\|\la(f)\|\leq C\|f\|_1\leq C \nu(\hg)^{1/2} \|f\|_2 =C \nu(\hg)^{1/2} \|f\|_{s,l},
\]
where the last inequality follows from Proposition 6.12 of \cite{folland-ra}.
\end{proof}

\begin{example}
\label{exa:com-cocompact-RD}
A class of examples satisfying the assumptions of the above proposition consists of Hecke pairs $\gh$, where $H$ contains a cocompact closed normal subgroup of $G$, see Example 4.3(ii) and Remark 6.3 of \cite{s4}.
\end{example}

There are several equivalent definitions for property (RD) which are easier to work with. In order to discuss them, we need some observations and notations.
\begin{remark}
\label{rem:equi-length-RD} Let $l_1$ and $l_2$ be two length functions on a Hecke pair $\gh$. If $l_1$ dominates $l_2$ and $\gh$ has (RD) with respect to $l_2$, then it has (RD) with respect to $l_1$ too. Therefore when $l_1$ and $l_2$ are equivalent, $\gh$ has (RD) with respect to $l_1$ if and only if it has (RD) with respect to $l_2$.
\end{remark}

Assume $l$ is a length function on a Hecke pair $\gh$. For every $r\geq 0$, we set
\[
B_{r,l}\gh:=\{xH\in \hg; l(x)\leq r\},
\]
\[
C_{r,l}\gh:=\{xH\in \hg; r\leq l(x)< r+1 \}.
\]
When $H$ is the trivial subgroup, we simply denote the above sets by $B_{r,l}(G)$ and $C_{r,l}(G)$, respectively. In the rest of this paper, the subscript $+$ in $\H_+\gh$, $L^2_+(\hg)$, etc., means that we are considering only non-negative real functions.
\begin{proposition}
\label{prop:RDdefinitions}
Let $l$ be a locally bounded length function on a Hecke pair $\gh$. If the left regular representation $\la:\H\gh\to B(L^2(\hg))$ is bounded, then the following conditions are equivalent:
\begin{itemize}
\item[(i)] The Hecke pair $\gh$ has property (RD) with respect to $l$.
\item[(ii)] There exists a polynomial $P$ such that for every $r>0$, if the support of a function $f\in \H_+\gh$ is contained in $B_{r,l}\gh$, then we have
    \[
    \|\la(f)\|\leq P(r)\|f\|_2.
    \]
\item[(iii)] There exists a polynomial $P$ such that for every $r>0$, if the support of a function $f\in \H_+\gh$ is contained in $B_{r,l}\gh$ and $\xi\in L^2_+(\hg)$, then we have
    \[
    \|f\ast \xi\|_2\leq P(r)\|f\|_2 \|\xi \|_2.
    \]
\end{itemize}
\end{proposition}
\begin{proof}
The proof of the above proposition is the same as Proposition 2.10 of \cite{s1}. We only prove the implication ``(ii)$\Rightarrow$ (i)'' to explain why the conditions of the locally boundedness of $l$ and the boundedness of the left regular representation are necessary.

Assume (ii) holds for some polynomial $P$. Without loss of generality, it is enough to prove (i) for any given $f\in \H_+\gh$. Since the support of $f$ is compact and $l$ is locally bounded, there is some $r>0$ such that $supp(f)\subseteq B_{r,l} \gh$. Find two positive numbers $C,s$ such that $P(n)\leq Cn^{s-1}$ for all $n\in \n$. For every $n\in \n$, let $\chi_n$ be the characteristic function of $C_{n-1,l}\gh$ which is a right $H$-invariant function. Since $\la$ is continuous, (ii) holds for every function in the closure of $\H_+\gh$ with respect to the $L^1$-norm. Therefore since $f\chi_n\in L^1\gh$ for all  $n\in \n$, we have
\[
\|\la(f\chi_n)\|\leq P(n)\|f\chi_n\|_2\leq C n^{s-1} \|f\chi_n\|_2.
\]
Thus we get
\[
\|\la(f)\|=\|\sum_{n=1}^\infty \la(f\chi_n)\|\leq C\sum_{n=1}^\infty n\inv n^{s} \|f\chi_n\|_2.
\]
By the Cauchy-Schwartz inequality, we obtain
\[
\|\la(f)\|\leq C' \left( \sum_{n=1}^\infty n^{2s} \|f\chi_n\|_2^2 \right)^{1/2},
 \]
where $C'=C \left( \sum_{n=1}^\infty n^{-2} \right)^{1/2}$. On the other hand, for every $n\in \n$ and $Hg\in C_{n-1,l}\gh$, we have
\[
\|f\chi_n\|_2^2= \int_{C_{n-1,l}\gh }|f(Hg)|^2 \d\nu(Hg),
\]
and also $n\leq l(g)+1$. Thus we obtain
\[
\|\la(f)\|\leq C' \left( \sum_{n=1}^\infty \int_{C_{n-1,l}\gh }|f(Hg)|^2 (l(g)+1)^{2s} \d\nu(Hg)  \right)^{1/2}=C'\|f\|_{s,l}.
\]
\end{proof}
When $\gh$ is a discrete Hecke pair, the above proof works even without using the fact that the left regular representation $\la$ is bounded. That is why the boundedness of $\la$ was not assumed in Proposition 2.10 of \cite{s1}. For the next theorem, we need to borrow Lemma 3.5 of \cite{chatpitsaloff}. However, we have to change its statement slightly. We also do not assume that $\eta$ is normalized, so various powers of $\eta(H)$ appear in our formulas.

\begin{lemma}
\label{lem:chat3-5} Let $H$ be a compact subgroup of a locally compact group $G$. For every $\xi\in L^2(G)$ and $f\in C_c(G)$, define $\, _H\xi\in L^2(G)$ and $\fh\in C_c(G)$ by
\begin{eqnarray*}
_H\xi(x)&:=&\left(\int_H|\xi (hx)|^2 d\eta(h)) \right)^{1/2},\\
\fh(x)&:=&\left(\int_H \int_H|f(hxk)|^2 d\eta(h)d\eta(k) \right)^{1/2},
\end{eqnarray*}
for all $x\in G$. Then we have the following statements:
\begin{itemize}
\item[(i)] $\|_H\xi\|_2= \eta(H)^{1/2} \|\xi\|_2$.
\item[(ii)] If $f$ is bi-$H$-invariant, then $|f \ast \xi(x)|\leq  |f| \ast\,  _H\xi (x)$ for all $x\in G$.
\item[(iii)] $\|\fh\|_2= \eta(H) \|f \|_2$.
\item[(iv)] $\|\la(f)\|\leq \eta(H) \|\la(\fh)\|$, where $\la$ is the left regular representation of the locally compact group $G$.
\end{itemize}
\end{lemma}
One notes that the item (ii) of the above lemma is proved by modifying a calculation in the proof of Lemma 3.5 of \cite{chatpitsaloff}. The following theorem is a crucial step in studying property (RD) of Hecke pairs $\gh$ when $H$ is a compact subgroup of $G$. First we need to recall some notations and formulas from Section 4 of \cite{s4}. For every $f\in C_c(H\ba G)$, we defined $\tilde{f}\in C_c(G)$ by setting $\tilde{f}(x):=f(Hx)$ for all $x\in G$. Then $\|\tilde{f}\|^2_2=\eta(H) \|f\|_2^2$, where the $L^2$-norms are taken in $L^2(G)$ and $L^2(H\ba G)$. It was also shown that, for every $f\in \H\gh$ and $g\in L^2(H\ba G)$, we have $\eta(H) \widetilde{f\ast g}=\tilde{f}\ast\tilde{g}$.

\begin{theorem}
\label{thm:compact-sub-RD}
Let $H$ be a compact subgroup of a locally compact group $G$. Then the group $G$ has (RD) if and only if the Hecke pair $\gh$ has (RD).
\end{theorem}
\begin{proof}
Suppose $G$ has (RD) with respect to a locally bounded length function $l$ on $G$. By Lemma \ref{lem:comp-length} and Remark \ref{rem:equi-length-RD}, without loss of generality, we can assume that the kernel of $l$ contains $H$, and so $l$ is a locally bounded length function on the Hecke pair $\gh$ as well. Hence $B_{r,l}\gh=H\ba B_{r,l}(G)$ for all $r\geq 0$. Let $P$ be the polynomial, mentioned in Proposition \ref{prop:RDdefinitions}(iii), coming from property (RD) for $G$. Given $r\geq 0$, for every function $f\in \H_+\gh$ such that $supp(f)\subseteq B_{r,l}\gh$ and every function $\xi\in L^2_+(\hg)$, let $\tilde{f}\in C_c(G)$ and $\tilde{\xi}\in L^2(G)$ be as above. Then $supp(\tilde{f})\subseteq B_{r,l}(G)$ and we have
\begin{eqnarray*}
\|f\ast \xi\|_2^2 &=& \frac{1}{\eta(H)^3}\|\tilde{f}\ast \tilde{\xi}\|_2^2\\
&\leq & \frac{P(r)^2}{\eta(H)^3}\|\tilde{f}\|_2^2 \| \tilde{\xi}\|_2^2\\
&= & \frac{P(r)^2}{\eta(H)}\|f\|_2^2 \| \xi\|_2^2.
\end{eqnarray*}
This proves that the Hecke pair $\gh$ has (RD) with respect to $l$.

Conversely, assume that the Hecke pair $\gh$ possesses (RD) with respect to a locally bounded length function $l$ and let $P$ be the corresponding polynomial. Given $r\geq 0$, assume that a function $f\in C_{c+}(G)$ subject to the condition $supp(f)\subseteq B_{r,l}(G)$ and a function $\xi\in L^2_+(G)$ are given. In the first step, we assume that $f$ is bi-$H$-invariant and $\xi$ is left $H$-invariant. We define $\hat{f}\in \H\gh$ and $\hat{\xi}\in L^2(\hg)$ by $\hat{f}(Hx):=f(x)$ and $\hat{\xi}(Hx):= \xi(x)$ for all $Hx\in \hg$, respectively. It is straightforward to check that
\[
\|\hat{f}\|_2^2=\frac{1}{\eta(H)}\|f\|_2^2, \qquad \|\hat{\xi}\|_2^2=\frac{1}{\eta(H)}\|\xi \|_2^2.
\]
For every $x\in G$, we compute
\begin{eqnarray*}
\widetilde{\hat{f}\ast \hat{\xi}} (x)&=&\int_\hg \hat{f}(Hxy\inv) \hat{\xi} (Hy)d\nu(Hy)\\
&=&\int_\hg f(xy\inv) \xi (y) d\nu(Hy)\\
&=& \frac{1}{\eta(H)} \int_\hg \left( \int_H f(xy\inv) \xi(y) d\eta(h)\right) d\nu(Hy)\\
&=& \frac{1}{\eta(H)} \int_\hg \left( \int_H f(xy\inv h\inv) \xi(hy) d\eta(h)\right) d\nu(Hy)\\
&=& \frac{1}{\eta(H)} \int_G f(xy\inv) \xi(y) d\mu(y)\\
&=&\frac{1}{\eta(H)}f\ast \xi (x).
\end{eqnarray*}
Since $supp(\hat{f})\subseteq B_{r,l}\gh$, we have
\begin{eqnarray*}
\|f\ast \xi\|_2^2 &=&\eta(H)^2\| \widetilde{\hat{f}\ast \hat{\xi}} \|_2^2\\
&=& \eta(H)^3 \| \hat{f}\ast \hat{\xi} \|_2^2\\
&\leq & \eta(H)^3 P(r)^2 \| \hat{f}\|_2^2 \|\hat{\xi} \|_2^2\\
&=&\eta(H) P(r)^2 \|f\|_2^2 \|\xi \|_2^2.
\end{eqnarray*}
Therefore the condition of Proposition \ref{prop:RDdefinitions}(iii) with the polynomial $\sqrt{\eta(H)}P(r)$ holds in this case. Now, let $f$ and $\xi$ be arbitrary (no $H$-invariance is assumed). Using Lemma \ref{lem:chat3-5} and the above calculation, we have
\begin{eqnarray*}
\|\fh\ast\xi\|_2^2&\leq& \|\fh \ast \, _H\xi\|_2^2 \\
&\leq& \eta(H) P(r)^2 \|\fh\|_2^2 \|_H\xi \|_2^2\\
&=&\eta(H)^2 P(r)^2 \|\fh\|_2^2 \|\xi \|_2^2.
\end{eqnarray*}
Hence,
\[
\|\la(\fh)\|\leq \eta(H) P(r) \|\fh\|_2=\eta(H)^2 P(r) \|f\|_2.
\]
By Lemma \ref{lem:chat3-5}(iv), this implies $\|\la(f)\|\leq \eta(H)^3 P(r) \|f\|_2$. Therefore, by Proposition \ref{prop:RDdefinitions}(ii), $G$ has (RD).
\end{proof}

The above theorem can be thought of as the continuous version of Theorem 2.11 of \cite{s1}. Since Theorem 2.2 of \cite{s2} (see also Corollary \ref{cor:commen-RD}) is also a generalization of the latter theorem, it would be interesting to state and prove the generalization of the above theorem for {\it compactly commensurable subgroups}, i.e. two closed subgroups whose intersection is a cocompact subgroup of both. So far in \cite{s4}, we have only considered compact and cocompact subgroups of locally compact groups to define non-discrete Hecke pairs. On the other hand, a compactly commensurable subgroup of a compact (resp. cocompact) subgroup is again a compact (resp. cocompact) subgroup. Hence, before generalizing the above theorem for compactly commensurable subgroups, we first need to define and study other non-discrete Hecke pairs besides the cases studied in \cite{s4}.

A version of the following corollary also appeared in \cite{chatpitsaloff} as Lemma 3.4.

\begin{corollary}
\label{cor:subcompactnormal}
Let $K$ be a compact normal subgroup of a locally compact group $G$. Then the group $G$ has (RD) if and only if the quotient group $G/K$ has (RD).
\end{corollary}

As another application of Theorem \ref{thm:compact-sub-RD}, we establish an equivalence between the property (RD) of a reduced discrete Hecke pair and the property (RD) of the totally disconnected locally compact group appearing in its Schlichting completion.
\begin{theorem}
\label{thm:RD-schlich}
Let $\gh$ be a reduced discrete Hecke pair and let $\ghb$ be its Schlichting completion. Then the followings are equivalent:
\begin{itemize}
\item[(i)] The Hecke pair $\gh$ has (RD).
\item[(ii)] The Hecke pair $\ghb$ has (RD).
\item[(iii)] The totally disconnected locally compact group $\gb$ has (RD).
\end{itemize}
\end{theorem}
\begin{proof}
Equivalence of (i) and (ii) follows from Remark \ref{rem:len-schlich-reduced}(ii) and Lemma \ref{lem:bijective-Schlich}. Equivalence of (ii) and (iii) follows from Theorem \ref{thm:compact-sub-RD}.
\end{proof}
For a similar result concerning the amenability and Haagerup property (a-T-menability) of Hecke pairs and their Schlichting completion see Proposition 5.1 of \cite{tzanev} and Proposition 4.5 of \cite{anan}, respectively. The following lemma which completes the above theorem follows from Remark \ref{rem:len-schlich-reduced}(i) and Proposition 2.12 of \cite{s4}:
\begin{lemma}
\label{lem:reduced-RD}
Let $\gh$ be a Hecke pair and let $N$ be a closed normal subgroup of $G$ which is contained in $H$. Put $G':=\frac{G}{N}$ and $H':=\frac{H}{N}$. Then the Hecke pair $\gh$ has (RD) if and only if the Hecke pair $(G',H')$ has (RD).

In particular, a discrete Hecke pair $\gh$ has (RD) if and only if its associated reduced discrete Hecke pair $\ghr$ has (RD).
\end{lemma}

\begin{remark}
\label{rem:newobst} The only known obstruction for a discrete group $G$ to possess property (RD) is that when it is amenable, property (RD) is equivalent to polynomial growth, see Corollary 3.1.8 of \cite{j2}, and also see Proposition \ref{prop:amenRD} for the same result in the setting of Hecke pairs. In addition to this obstruction, unimodularity is another necessary condition for a locally compact group to possess property (RD), see Theorem 2.2 of \cite{JiSch}. The following corollary shows that the similar notion of relative unimodularity is necessary for a discrete Hecke pair to possess property (RD).

This suggests that the theory of locally compact groups can provide a better understanding of Hecke pairs. This have been the main point of several papers which use the Schlichting completion to analyze Hecke pairs and Hecke \cs-algebras, see for example \cite{tzanev, klq, anan} and our recent works \cite{s3, s4}. For another application of the Schlichting completion in the present paper see Proposition \ref{prop:f-ext-dis}.
\end{remark}

\begin{corollary}
\label{cor:unimod-RD}
If a discrete Hecke pair $\gh$ has (RD), then it must be relatively unimodular.
\end{corollary}
\begin{proof}
Using Lemma \ref{lem:reduced-RD}, without loss of generality, we can restrict our attention to the case that $\gh$ is a reduced discrete Hecke pair. If $\D_\gh\neq 1$, then the locally compact group $\gb$ in the Schlichting completion of $\gh$ is not unimodular and therefore cannot have (RD) by Theorem 2.2 of \cite{JiSch}. It follows from this and Theorem \ref{thm:RD-schlich} that the discrete Hecke pair $\gh$ cannot possess (RD).
\end{proof}
\begin{example}
\label{exa:BC-HP-notRD}
Consider the Bost-Connes Hecke pair $\gh$;
\[
G=\left\{ \left( \begin{array} {rr}1&b\\0&a \end{array}\right); a\in \q^+, b\in \q \right\}, \quad \text{and} \quad H=\left\{ \left( \begin{array} {rr}1&n\\0&1
\end{array}\right); n\in \z \right\}.
\]
Since it is not relatively unimodular, see for instance 2.1.1.3 of \cite{hall}, it does not have property (RD). The same statement holds for all Hecke pairs appearing in the context of quantum dynamical systems associated with Hecke pairs which imitate the Bost Connes construction in \cite{bc}. This is due to the fact that those dynamical systems are built on the relative modular functions of the underlying Hecke pairs which are required to be non-trivial.
\end{example}
As another application of Theorem \ref{thm:RD-schlich}, we improve Corollary 3.7 of \cite{s2} as follows:

\begin{proposition}
\label{prop:f-ext-dis} Let $\gh$ be a discrete Hecke pair and let $\Gamma$ be a group containing $G$ as a closed finite index subgroup such that $(\Gamma, H)$ is a discrete Hecke pair too. Then the Hecke pair $\gh$  has (RD) if and only if the Hecke pair $(\Gamma,H)$ has (RD).
\end{proposition}
\begin{proof}
If the Hecke pair $(\Gamma, H)$ has (RD), then by Proposition 2.11 of \cite{s2}, the Hecke pair $\gh$ has (RD) too.

Conversely, suppose that the Hecke pair $\gh$ has (RD). By Lemma \ref{lem:reduced-RD}, without loss of generality, we can assume that the Hecke pair $(\Gamma,H)$ is reduced. Let $(\overline{\Gamma}, \hb)$ be the Schlichting completion of $(\Gamma,H)$ and let $\gb^{\Gamma}$ be the closure of $G$ in $\overline{\Gamma}$. By the same argument as the proof of Theorem \ref{thm:RD-schlich} and using Lemma \ref{lem:bijective-Schlich}, we deduce that the Hecke pair $(\gb^{\Gamma}, \hb)$ has (RD). Since $\hb$ is a compact subgroup of $\gb^\Gamma$, by Theorem \ref{thm:compact-sub-RD}, the locally compact group $\gb^\Gamma$ has (RD). Since $\gb^\Gamma$is a closed finite index subgroup of $\overline{\Gamma}$, by Lemma 3.3 of \cite{chatpitsaloff}, the locally compact group $\overline{\Gamma}$ has (RD). Therefore by Theorem \ref{thm:RD-schlich}, the Hecke pair $(\Gamma,H)$ has (RD).
\end{proof}
There are two points about the above proposition: First, Lemma 3.3 of \cite{chatpitsaloff} is still true without assuming that $G$ is a compactly generated group. Secondly, it is necessary to assume that $(\Gamma,H)$ is a Hecke pair, see the following example:
\begin{example}
\label{exa:find-nothp} Consider the direct product $G=H\times K$ of two copies of $\z$ and assume that $t$ and $s$ are generators of $H$ and $K$, respectively. Let $\Gamma$ be a torsion version of HNN extension of $G$ as follows:
\[
\Gamma:=\langle t,s,u; utu\inv=s, u^2=e, st=ts \rangle.
\]
Clearly, $\gh$ is a Hecke pair and $\Gamma$ contains $G$ as a finite index subgroup, but one easily observes that $H$ is not a Hecke subgroup of $\Gamma$.
\end{example}
Now, we are ready to illustrate the application of our results in proving property (RD) for specific discrete Hecke pairs.
\begin{example}
\label{exa:sl2-z1p} It is proved in Th\'{e}or\`{e}me 2(3) of \cite{jv} that if $G$ is a unimodular locally compact group acting properly on a locally finite tree $\xx$ with finite quotients, then $G$ has property (RD). A famous example for this type of groups is the group $SL_2(\q_p)$. Thus by theorem \ref{thm:compact-sub-RD} and Lemma \ref{lem:reduced-RD}, the Hecke pair $(PSL_2(\q_p), PSL_2(\z_p))$ has property (RD). Then it follows from Theorem \ref{thm:RD-schlich} that the Hecke pair $(PSL_2(\z[1/p]), PSL_2(\z))$ has (RD), see Example \ref{exa:SL2zp}.
\end{example}

In order to investigate more examples of Hecke pairs possessing property (RD), we mainly use Theorems \ref{thm:compact-sub-RD} and \ref{thm:RD-schlich}, the main result of I. Chatterji, C. Pittet and L. Saloff-Coste in \cite{chatpitsaloff}, and the main result S. Mustapha in \cite{mustapha}. These latter results are summarized in the following remark:

\begin{remark}
\label{rem:RD-groups}
\begin{itemize}
\item [(i)] Let $G$ be a connected Lie group. Let $\widetilde{G}$ and $\mathfrak{g}$ denote the universal cover of $G$ and the Lie algebra of $G$, respectively. Then the main result (Theorem 0.1) of \cite{chatpitsaloff} asserts that the following statements are equivalent:
    \begin{itemize}
\item [(a)] The group $G$ has (RD).
\item [(b)] The Lie algebra $\mathfrak{g}$ is a direct product of a semisimple Lie algebra and a Lie algebra of type $R$.
\item [(c)] The Lie group $\widetilde{G}$ is a direct product of a connected semisimple Lie group and a Lie group of polynomial growth.
    \end{itemize}
    We recall that a Lie algebra is called of {\it type R} if all the weights of its adjoint representation are purely imaginary. Then it is known that a Lie algebra is of type R if and only if its associated Lie group is of polynomial growth, see \cite{guiv, jen}. This explains the equivalence of Conditions (b) and (c) in the above. Also, for the similar statement about polynomial growth of $p$-adic Lie groups and several other equivalent conditions see \cite{raja}.
\item [(ii)] Let $F$ be a local field of characteristic $0$ and let $G$ be an algebraic group over $F$. Assume $G$ and its radical are both compactly generated. Then $G$ has property (RD) if and only if $G$ is a reductive group. For a more general statement see Theorem 1 of \cite{mustapha}. We note that it was also shown in Theorem 4.5 of \cite{chatpitsaloff} that every semisimple linear algebraic group on a local field has property (RD).
\end{itemize}
\end{remark}

\begin{example}
\label{exa:sl2-RD}
The above discussion applies to the semisimple Lie group $SL_2(\r)$, and therefore the Hecke pair $(SL_2(\r), SO_2(\r))$ discussed in \cite{s4} has property (RD).
\end{example}

Using Proposition \ref{prop:normalcomponent} and Theorem \ref{thm:compact-sub-RD}, a locally compact Hecke pair $\gh$ has property (RD) provided that there exists a normal subgroup $K$ of $G$ such that $K\subseteq H$, $H/K$ is compact and $G/K$ is of one of the forms described in the above remark. This is particularly useful to find pro-Lie groups with property (RD). In the next subsection, we expand on the notion of growth of a Hecke pair.

\section{Growth, amenability and property (RD)}
\label{sec:growthRD}

In addition to compactly (finitely) generated groups, we might also deal with not necessarily compactly (possibly infinitely) generated groups and Hecke pairs associated with them. Thus we need to extend some notions from the setting of compactly generated groups to this class of groups and Hecke pairs. To this end, we have to study the growth rates of groups with respect to arbitrary locally bounded length functions (not just length functions associated with compact generating sets). Another motivation for our general approach is that word length functions are not well defined on compactly generated Hecke pairs.

\begin{definition} Let $\gh$ be a Hecke pair and let $l$ be a locally bounded length function on $\gh$.
\begin{itemize}
\item [(i)] The {\it growth function associated with $l$} is the function $\G_l:[0,\infty[ \ra [0,\infty]$ defined by $\G_l(r):=\nu (B_{r,l}\gh)$ for all $r\geq 0$.
\item [(ii)] We say that $\gh$ has {\it infinite growth \wrt $l$} if $\G_l(r)=\infty$ for some $r\geq 0$, otherwise we say that $\gh$ has a {\it finite growth  \wrt $l$}.
\item [(iii)] We say that $\gh$ is of {\it polynomial growth \wrt $l$} if there are two positive constants $c, \alpha$ such that $\G_l(r)\leq cr^\alpha$ for all large enough real numbers $r$. In this case we also say that the {\it degree of the growth of $\gh$} is at most $\alpha$.
\item [(iv)] We say that $\gh$ is of {\it superpolynomial growth \wrt $l$} if its growth is faster than any polynomial, more precisely, if the limit $\lim_{r\to \infty}\frac{\ln \G_l(r)}{\ln r}$ exists and equals $\infty$.
\item [(v)] We say that $\gh$ is of {\it exponential growth \wrt $l$} if it has a finite growth and if there are two positive constants $d, \beta$ such that $\G_l(r)\geq d\beta^r$ for all sufficiently large real numbers $r$.
\item [(vi)] We say that $\gh$ is of {\it subexponential growth \wrt $l$} if it has a finite growth \wrt $l$, but it is not of exponential growth \wrt $l$. It is equivalent to the condition that $\lim_{r\to \infty}\frac{\ln \G_l(r)}{ r}=0$.
\item [(vii)] We say that $\gh$ is of {\it intermediate growth \wrt $l$} if it is of superpolynomial and subexponential growth \wrt $l$, simultaneously.
\end{itemize}
\end{definition}

There are several points regarding the above definition: When a Hecke pair $\gh$ is of polynomial (resp. intermediate or exponential) growth with respect to some length function $l$, we may also say that $\gh$ has {\it polynomial} (resp. {\it intermediate} or {\it exponential}) {\it growth rate} with respect to $l$. When $\gh$ is a discrete Hecke pair $\nu$ is a multiple of the counting measure. Assuming $H=\{e\}$, the above definitions reduce to the well known definitions of various growth rates for locally compact groups. It is clear that equivalent length functions give rise to the same growth rate. Therefore, when a group $G$ is a compactly (or finitely) generated group, we usually consider the growth rate of $G$ with respect to the word length function associated with a compact generating set and do not mention the length function. Similar to amenability and property (RD), (see \cite{tzanev, s3} for amenability and see Remark \ref{rem:len-schlich-reduced}, Theorem \ref{thm:RD-schlich} and Lemma \ref{lem:reduced-RD} for property (RD)), certain operations on Hecke pairs do not change the growth rate. Here, we only mention them for later references and skip routine proofs.

\begin{lemma}
\label{lem:growth-normal-component} Let $\gh$ be a Hecke pair equipped with a length function $l$. Let $N$, $G'$ and $H'$ be as in Proposition \ref{prop:normalcomponent} and let $l'$ be the length function defined on the Hecke pair $\ghpr$ using $l$, see Remark \ref{rem:len-schlich-reduced}(i). Then the Hecke pairs $\gh$ and $\ghpr$ have the same growth rate with respect to $l$ and $l'$, respectively.
\end{lemma}
The special case of the above lemma that $\gh$ is a discrete Hecke pair and $N=K_\gh$ is particularly important in the study of growth rate and amenability of discrete Hecke pairs.
\begin{proposition}
\label{prop:growth-schlich} Let $\gh$ be a reduced discrete Hecke pair. Let $l$ be a length function on $\gh$ and let $\bar{l}$ be the corresponding length function on its Schlichting completion $\ghb$ defined using $l$, see Remark \ref{rem:len-schlich-reduced}(ii). Then the discrete Hecke pairs $\gh$ and $\ghb$ have the same growth rate with respect to the length functions $l$ and $\bar{l}$, respectively.
\end{proposition}

Using the above results, we obtain a criterion for amenability of discrete Hecke pairs in the next proposition. First we recall the definition of amenability of pairs $\gh$ based on \cite{eymard}, see also \cite{tzanev, anan, s3} for more details and results concerning amenability of Hecke pairs.

\begin{definition}
\label{def:amen-pair}
Let $H$ be a closed subgroup of a locally compact group $G$. The pair $\gh$ is called amenable if it possesses the fixed point property, that is whenever $G$ acts continuously on a compact convex subset $Q$ of a locally convex topological vector space by affine transformations and the restriction of this action to $H$ has a fixed point, the action of $G$ has a fixed point too.
\end{definition}

\begin{proposition}
\label{prop:growthamen}
Let $\gh$ be a discrete Hecke pair which has a finite generating set $S$. If the Hecke pair $\gh$ is of subexponential growth, then it is amenable.
\end{proposition}
\begin{proof}
Assume that $\gh$ is of subexponential growth, then so is its reduction $\ghr$. Thus the Schlichting completion $(\overline{G_r},\overline{H_r})$ of the latter Hecke pair is of subexponential growth. Since $\overline{H_r}$ is compact, it follows that $\overline{G_r}$ is of subexponential growth with respect to a length function defined by a compact generating set. It is a well known fact that compactly generated locally compact groups of subexponential growth are amenable, see \cite{guiv}. Therefore $\overline{G_r}$ is amenable. Using Proposition 5.1 of \cite{tzanev}, this implies that the Hecke pair $\gh$ is amenable.
\end{proof}

Due to the fact that every compactly generated locally compact group of subexponential growth is unimodular, see Lemma I.3 of \cite{guiv}, and by applying a similar argument as the above proof, one can prove the following proposition:

\begin{proposition}
\label{prop:growth-rel-unimod}
Let $\gh$ be a discrete Hecke pair which has a finite generating set $S$. If the Hecke pair $\gh$ is of subexponential growth, then it is relatively unimodular.
\end{proposition}

Since, in the setting of finitely generated discrete Hecke pairs, polynomial growth implies property (RD) and property (RD) requires relative unimodularity, the above conclusion can be reached using the next proposition as well. First, we need to recall and use a characterization of amenability.

\begin{remark}
\label{rem:kesten-amen}
A locally compact group $G$ is amenable if and only if for every non-negative real function $f\in L^1(G)$, we have $\|f\|_1=\|\la(f)\|$, see \cite{leptin}. Using this, one can show that every compactly generated amenable group with property (RD) is of polynomial growth. Conversely, the continuity of the left regular representation implies that every compactly generated locally compact group of polynomial growth possesses property (RD). For the proof of these statements see Theorem 1.5 of \cite{chatru}.
\end{remark}

Using the above remark, Theorem \ref{thm:RD-schlich}, Proposition \ref{prop:growth-schlich}, and Remark 6(iv) of \cite{s3}, we obtain the following proposition:

\begin{proposition}
\label{prop:amenRD}
Let $\gh$ be a finitely generated discrete Hecke pair.
\begin{itemize}
\item [(i)] If $\gh$ is of polynomial growth, then it has (RD).
\item [(ii)] If $\gh$ is amenable and possesses property (RD), it is of polynomial growth.
\end{itemize}
\end{proposition}

One notes that using Theorem \ref{thm:compact-sub-RD} and Theorem 4.1 of \cite{s4}, the above proposition is still valid for non-discrete Hecke pairs $\gh$, where $G$ is a compactly generated locally compact group and $H$ is a compact subgroup of $G$. When $G$ is a unimodular locally compact group and it has a decomposition of the form $G=PK$, where $K$ is compact and $P$ is amenable, a notably stronger result is also available, see Theorem 4.4 and Proposition 4.2 of \cite{chatpitsaloff}.


\bibliographystyle{amsalpha}

\end{document}